\documentclass[11pt]{amsart}
\usepackage{bm}
\usepackage{graphicx}
\usepackage{mathpazo}
\usepackage{latexsym}   
\usepackage{amssymb}    
\usepackage{amsmath}    
\usepackage{amsthm}     
\usepackage{math}
\usepackage{xspace}
\usepackage[margin=1in]{geometry}
\usepackage{hyperref}

\usepackage{diagrams}
\newarrow{to}---->
\newarrow{dashedto}{}{dash}{}{dash}>
\newarrow{mto}|--->
\newarrow{impto}===={=>}
\newarrowtail{<=}\Rightarrow\Leftarrow\Uparrow\Downarrow
\newarrow{bboth}{<=}==={=>}
\newarrow{inject}{hooka}---{vee}
\newarrow{surject}----{>>}

\def\satotate{{\rm ST}}

\numberwithin{equation}{section}
\numberwithin{table}{section}

\newtheorem{theorem}{Theorem}[section]
\newtheorem{lemma}[theorem]{Lemma}
\newtheorem{proposition}[theorem]{Proposition}
\newtheorem{corollary}[theorem]{Corollary}

\theoremstyle{definition}
\newtheorem{example}[theorem]{Example}

\newcommand{\stk}[1]{{\mathcal #1}}

\def\setcomp{\smallsetminus}
\def\one{{1}}
\def\GSp{\gsp}
\def\cent{{\mathcal Z}}
\def\cclass{{\mathcal C}}
\def\cc{\iota}
\def\O{{\mathcal O}}

\def\et{{\rm et}}
\def\tor{{\rm tor}}

\newcommand{\set}[1]{\left\{#1\right\}}
\newcommand{\gen}[1]{\left\langle#1\right\rangle}

\DeclareMathOperator{\GL}{GL}
\DeclareMathOperator{\Split}{Split}
\DeclareMathOperator{\Gal}{Gal}
\DeclareMathOperator{\Frob}{Frob}
\DeclareMathOperator{\cond}{cond}

\newcommand{\Z}{\mathbb Z}
\newcommand{\Q}{\mathbb Q}
\newcommand{\F}{\mathbb F}

\renewcommand{\split}{\textbf{Split}\xspace}
\newcommand{\dqs}{\textbf{DQ-S}\xspace}
\newcommand{\dqi}{\textbf{DQ-I}\xspace}
\newcommand{\quartic}{\textbf{Quartic}\xspace}
\newcommand{\qrl}{\textbf{QRL}\xspace}
\newcommand{\drls}{\textbf{DRL-S}\xspace}
\newcommand{\drli}{\textbf{DRL-I}\xspace}

\title{Local heuristics and an exact formula for abelian surfaces over finite fields}

\author{Jeffrey D. Achter}
\address{Colorado State University, Fort Collins, CO 80523-1874}
\email{achter@math.colostate.edu}
\urladdr{http://www.math.colostate.edu/~achter}
\thanks{JDA was partially supported by grants from the Simons
  Foundation (204164) and the NSA (H98230-14-1-0161).}

\author{Cassandra Williams}
\address{James Madison University, Harrisonburg, VA 22807}
\email{willi5cl@jmu.edu}
\urladdr{http://educ.jmu.edu/~willi5cl}

\begin{document}

\begin{abstract}
Consider a quartic $q$-Weil polynomial $f$.  Motivated by equidistribution
considerations we define, for each prime $\ell$, a local factor which
measures the relative frequency with which $f\bmod \ell$ occurs as the
characteristic polynomial of a symplectic similitude over $\ff_\ell$.  For a certain
class of polynomials, we show that the resulting infinite product
calculates the number of principally polarized abelian surfaces over $\ff_q$
with Weil polynomial $f$.
\end{abstract}

\maketitle

\section{Introduction}

Consider abelian varieties over a finite field.  To each such $X/\ff_q$ one may associate a characteristic polynomial of Frobenius, $f_{X/\ff_q}(T) \in \integ[T]$; and two abelian varieties $X$ and $Y$ are isogenous if and only if $f_{X/\ff_q}(T) = f_{Y/\ff_q}(T)$.  In this way, isogeny classes of abelian varieties over $\ff_q$ are parametrized by suitable $q$-Weil polynomials $f(T)$.

Conversely, given such a polynomial $f$, it is of intrinsic interest to calculate how many abelian varieties are in the corresponding isogeny class.  In fact, a polarized variant of this problem seems even more natural.  Let $\stk A_g$ be the moduli space of principally polarized abelian varieties of dimension $g$, and let
\[
\stk A_g(\ff_q;f) = \st{(X,\lambda) \in \stk A_g(\ff_q): f_{X/\ff_q}(T) = f(T)}.
\]
Armed with an (overly optimistic) equidistribution philosophy, one might attempt to estimate $\#\stk A_g(\ff_q;f)$ in the following fashion.  To the extent possible, Frobenius elements of abelian varieties are equidistributed in $\gsp_{2g}(\ff_\ell)$.  By somehow multiplying together, over all $\ell$, the frequency with which $f(T)\bmod \ell$ occurs as the characteristic polynomial of a symplectic similitude, one might try to apprehend $\#\stk A_g(\ff_q;f)$.

As written, this strategy is nonsense; for given $g$ and $\ell$, $\bmod \:\ell$ Frobenius elements are equidistributed only if $q \gg_g \ell$.  Nonetheless, such congruence considerations apparently control the sizes of isogeny classes.  

Our main result is as follows.  For $f$ in a certain class of simple, ordinary $q$-Weil polynomials of degree $4$ (see Section \ref{secweilpolys}), we define for each prime $\ell$ a quantity $\nu_\ell(f)$ (see \eqref{eqdefnuell}) which measures its relative frequency as the characteristic polynomial of an element $\gsp_4(\ff_\ell)$.  After defining a Sato-Tate term $\nu_\infty(f)$, we show that
\begin{equation}
\label{eqintrotheorem}
\nu_\infty(f)\prod_\ell \nu_\ell(f) = \#\stk A_2(\ff_q;f)
\end{equation}
(where a principally polarized abelian surface is given mass inversely proportional to the size of its automorphism group).

The present work is inspired by work of Gekeler \cite{gekeler03}, who derived a version of \eqref{eqintrotheorem} for elliptic curves over a finite prime field.  Our perspective was influenced by Katz's analysis \cite{katz-lt} of Gekeler's product formula.

\subsection*{Acknowledgments}
We thank Everett Howe for sharing with us his work on principally
polarized abelian varieties within a given isogeny class.  We further
thank the referee for helpful suggestions.

\section{Abelian varieties and Weil polynomials}\label{secweilpolys}

Let $X/\ff_q$ be an abelian variety of dimension $g$ over a finite
field with $q = p^e$ elements, and let
$f_{X/\ff_q}(T)\in \integ[T]$ be the characteristic polynomial of its Frobenius
endomorphism (acting on, say, any of the Tate modules $T_\ell X$ with
$\ell\not = p$).  Then $f_{X/\ff_q}(T)$ is a $q$-Weil polynomial,
i.e., the complex roots $\alpha_1, \cdots,
\alpha_{2g}$ of $f_{X/\ff_q}(T)$ may be ordered so that $\alpha_j
\alpha_{g+j} = q$ for $1 \le j \le g$, and in fact $\abs{\alpha_j} = \sqrt q$ for each $j$.  

Now assume that $f(T)$ is a $q$-Weil polynomial of degree $4$; such a
polynomial corresponds to a (possibly empty) isogeny class $\cali_f$ of abelian
surfaces over $\ff_q$.  In the
sequel, we will assume that $f$ is:
\begin{enumerate}\def\theenumi{W.\arabic{enumi}}
\item {\em ordinary} \label{enord} its middle coefficient is relatively
  prime to $p$;
\item {\em principally polarizable} \label{enpp} there exists a principally polarized
  abelian surface with characteristic polynomial $f$;
\item{\em Galois}\label{engal}  the polynomial $f(T)$ is irreducible over $\rat$, and
  $K_f := \rat(T)/f(T)$ is Galois and unramified at $p$;
\item{\em maximal}\label{enmax} let $\varpi_f$ be a (complex) root of $f(T)$,
  with complex conjugate $\bar\varpi_f$.  Then $\calo_f :=
  \integ[\varpi_f,\bar\varpi_f]$, {\em a priori} an order in $K_f$, is
  actually the maximal order $\calo_{K_f}$.
\end{enumerate}
Conditions \ref{enord} and \ref{engal} imply that any abelian
surface in $\stk A_2(\ff_q;f)$ is ordinary and simple.
Condition \ref{enpp} is explicitly characterized, in terms of the coefficients
of $f$, in \cite[Thm.\ 1.3]{howe95}.  Condition \ref{enmax} is indeed an
extra hypothesis, which we hope to relax in a future work. The
isomorphism class of $\calo_f$, as an abstract order, is independent of the
choice of $\varpi_f$.

Note that $\gal(K_f/\rat)$ is abelian, since $[K_f:\rat]=4$; and there is an intrinsically defined complex conjugation $\iota \in \gal(K_f/\rat)$, since $K_f$ is a CM field. As in the description of condition \ref{enmax}, we will often denote the action of $\iota$ on an element $\alpha \in K_f$ by $\bar\alpha = \iota(\alpha)$. If $R$ is any ring, then $\iota$ acts on $\calo_{K_f}\tensor R$ via the first component.

\begin{example}
The polynomial $f(T) = T^{4} + 29T^{3} + 331T^{2} + 1769T + 3721$ is a
61-Weil polynomial which is ordinary, principally polarizable, Galois
and maximal.  In fact, $\calo_f$  is the ring of
integers in $\rat(\zeta_5)$, and $f$ is the characteristic polynomial of
Frobenius of the Jacobian of the curve with affine equation $y^2 =
x^5-2$.
\end{example}

Define the conductor of $f$, $\cond(f)$, as the index of
$\integ[\varpi_f]\iso \integ[T]/f(T)$ in $\calo_f$.  If $f(T) = T^4-aT^3+bT^2-aTq+q^2$, let
$f^+(T) = T^2 - aT+(b-2q)$; then $f^+(T)$ is the minimal polynomial
of $\varpi_f + \bar\varpi_f$, and $K^+_f := \rat[T]/f^+(T)$ is
the maximal totally real subfield of $K_f$.  Denote the discriminants
of $f(T)$ and $f^+(T)$ by $\Delta_f$ and $\Delta_{f^+}$, respectively. 
Similarly, let $\Delta_{\O}$ represent the discriminant of an order $\O$; notice that $\Delta_{\Z[\varpi_f]}=\Delta_f$ and $\Delta_{\O_{K^+}}=\Delta_{f^+}$.

\begin{lemma}
\label{lemconductorq}
The index of $\integ[\varpi_f]$ in $\calo_f$ is $q$.
\end{lemma}

\begin{proof}
Using the above definition of $f$ and Propositions 9.4 and 9.5 of \cite{howe95}, 
$$\Delta_{\calo_f}=\Delta_{f^+}^2 \cdot N_{K_f/\rat}(\varpi_f-\bar\varpi_f)= (a^2-4b+8q)^2(b^2+4bq+4q^2-4a^2q).$$
The discriminant of $\Z[\varpi_f]$ is given by
$$\Delta_{\integ[\varpi_f]}=\Delta_f=q^2(a^2-4b+8q)^2(b^2+4bq+4q^2-4a^2q)=q^2 \Delta_{\calo_f}$$
and $\Delta_f=[\O_f:\integ[\varpi_f]]^2\Delta_{\calo_f}$.
Then the desired index is $q$.
\end{proof}

\begin{corollary}
\label{lemalmostmaximal}
If $\ell\not = p$, then $\calo_{K_f}\tensor \integ_{(\ell)} \iso \integ_{(\ell)}[T]/f(T)$.
\end{corollary}

Similarly, $\Z[T]/f^+(T)$ is maximal:

\begin{lemma}
\label{lemmaximalrealorder}
The order $\Z[T]/f^+(T)$ is the maximal order $\O_{K_f^+}$.
\end{lemma}
\begin{proof}
Condition \ref{enmax} implies that $\O_f\cap K_f^+=\O_{K_f}\cap K_f^+=\O_{K_f^+}$.  Certainly $\Z[T]/f^+(T)=\Z[\varpi_f+\bar\varpi_f] \subseteq \O_f\cap K_f^+$.
Consider $a\in\O_f$; then $a=a_0+a_1\varpi_f+a_2\bar\varpi_f + a_3\varpi_f\bar\varpi_f$ for some integers $a_i$.  We have $\varpi_f\bar\varpi_f=q$ as $f$ is a $q$-Weil polynomial, and $a\in K_f^+$ if and only if $a_1=a_2$.  Then $a\in\O_f\cap K_f^+$ has the form $a=(a_0+a_3q)+a_1(\varpi_f+\bar\varpi_f)$ and $\O_f\cap K_f^+\subseteq \Z[\varpi_f+\bar\varpi_f]$. Thus, $\Z[\varpi_f+\bar\varpi_f] = \O_{K_f^+}$.
\end{proof}


\section{Conjugacy classes in symplectic groups}

If $X/\ff_q$ is a principally polarized abelian surface, then the
four-dimensional $\ff_\ell$-vector space $X_\ell := X[\ell](\bar\ff_q)$
is naturally equipped with a symplectic form.  We collect some
notation concerning symplectic (similitude) groups. 

\subsection{Symplectic groups}

Let $V$ be a vector space of dimension $2g$ over a field $k$, equipped
with a perfect, skew-symmetric form $\ang{\cdot,\cdot}$.  The
symplectic similitude group of $V$ is the group of automorphisms which
preserve this form up to a multiple.  Concretely,
\[
\gsp(V,\ang{\cdot,\cdot})  = \st{\gamma \in \gl(V): \exists m(\gamma)
  \in k\units: \forall u,v \in V, \ang{\gamma u, \gamma v} = m(\gamma)
  \ang{u,v}}.
\]
The group of automorphisms of the symplectic space is the symplectic
group, $\SP(V,\ang{\cdot,\cdot})$, and these groups sit in an exact sequence
\begin{diagram}[LaTeXeqno]\label{diagsp}
1 & \rto & \SP(V,\ang{\cdot,\cdot}) & \rto & \gsp(V,\ang{\cdot,\cdot})
& \rto^{{\rm mult}} & k\units & \rto & 1 \\
&&&&\gamma & \rmto & m(\gamma).
\end{diagram}
For $m \in k\units$, we let $\gsp(V,\ang{\cdot,\cdot})^{(m)} = 
{\rm mult}\inv(m)$.

Call a decomposition $V = W_1\oplus W_2$ symplectic if, for each $i$,
$\ang{\cdot,\cdot}\rest{W_i}$ is a perfect pairing; and isotropic if,
for some $i$, $\ang{\cdot,\cdot}\rest{W_i} = 0$.

In fact, any symplectic space $V$ of dimension $2g$ is isomorphic to
$k^{\oplus 2g}$, equipped with the pairing described by the $2g\times
2g$ matrix
\[
J = \begin{pmatrix}
0 & \one_g \\
-\one_g & 0
\end{pmatrix};
\]
the associated similitude and symplectic groups are $\gsp_{2g}(k)$ and
$\SP_{2g}(k)$, respectively.

\subsection{Shapes of conjugacy classes}\label{secclassshapes}

In a general linear group $\gl(V)$, semisimple conjugacy classes are
parametrized by the theory of rational canonical form (RCF), which gives a decomposition of any automorphism of a vector space into a direct sum of cyclic automorphisms over invariant subspaces.  (An automorphism is cyclic if and only if its minimal and characteristic polynomials coincide.)  Specifically, we factor the characteristic polynomial of $\gamma$ into a product of irreducible polynomials
$$f_\gamma(T)=\prod_i \left(\phi_i(T)\right)^{\lambda_i}$$
 and then associate to each factor $\phi_i$ a partition of its multiplicity $\lambda_i$.  Denote the partition by $[\lambda_{i,1}, \lambda_{i,2}, \ldots, \lambda_{i,n}]$ where $\lambda_{i,1}\geq \lambda_{i,2}\geq \dots \geq \lambda_{i,n}$.  Then $V$ has a $\gamma$-invariant subspace with characteristic polynomial $\phi_i^{\lambda_{i,j}}$ for each $1\leq j\leq n$, and $\gamma$ restricted to each of these subspaces is cyclic.  Note that the minimal polynomial of $\gamma$ is the product of $\phi_i^{\lambda_{i,1}}$, so $\gamma$ is cyclic if and only if the partition of $\lambda_i$ consists of a single part for each $\phi_i$.
Thus, arbitrary conjugacy
classes in $\gl_{2g}(k)$ are  determined by their characteristic polynomial and
additional partition data.

The classification of conjugacy classes in $\gsp_{2g}(k)$ is
more intricate for two reasons.  First, for elements with repeated eigenvalues, the
presence of the symplectic form places nontrivial restrictions on
allowable partition data.  Second, elements of $\gsp_{2g}(k)$
which are conjugate in $\gl_{2g}(k)$ need not be conjugate in the
symplectic similitude group; certain $\gl_{2g}$-conjugacy classes
decompose into classes indexed by $k\units/(k\units)^2$.  
(For details of this decomposition, see, for example, \cite{fulman00} or \cite{wall}.  Alternatively, compare our results to those of \cite{breeding} or \cite{shinoda82}.)

In the sequel, we will only need the case where $g=2$ and $k
=\ff_\ell$ is a finite field. 
Let $\calc(\gamma)$ denote the conjugacy class of $\gamma$.  We  distinguish 
conjugacy classes by the factorization pattern of their characteristic polynomials $f_\gamma(T)$ into irreducible polynomials over $\ff_\ell$, and then refine this with additional combinatorial data, if necessary.  The resulting collection of data associated to $\calc(\gamma)$ will be called the {\em shape} of $\gamma$, or of its conjugacy class.

Our enumeration of conjugacy classes in $\gsp_4(\ff_\ell)$ is purposefully incomplete; we only include those which arise in our subsequent study of abelian surfaces.  First, we only consider those classes for which all irreducible factors of the characteristic polynomial have the same degree.  (Briefly call such a class "relevant".)  Second, we only list those conjugacy classes corresponding to regular, or cyclic, elements. In general, an element of an algebraic group $\gamma \in G(k)$ is called regular if the dimension of its centralizer, $\dim \cent_G(\gamma)$, is minimal, i.e., equal to the rank of $G$.  In the case of $G = \gsp_4$, it is equivalent to insist that $\gamma$ be cyclic in the standard representation, i.e., that there exists $v\in V$ such that $\st{\gamma^iv : i\ge 0}$ spans $V$.  
  As usual, a semisimple element is regular if and only if its eigenvalues are distinct.

Let $f_\gamma(T) = \prod_j g_j(T)^{e_j}$ be the factorization of $f_\gamma(T)$ into powers of distinct, irreducible monic polynomials of equal degree.  To this factorization of $f_\gamma(T)$ there is an associated factorization $V \iso \oplus W_j$, where $\gamma\rest{W_j}$ has characteristic polynomial $g_j(T)^{e_j}$.  For each $j$, either $\ang{\cdot,\cdot}\rest{W_j}$ is zero or it is perfect; call these factorizations isotropic and symplectic, respectively.

\noindent\textbf{Case 1: Regular semisimple elements}


A regular semisimple conjugacy class is one for which the elements have a squarefree characteristic polynomial.  
We classify such conjugacy classes by the factorization of $f_\gamma(T)$ (over $\ff_\ell$) and by $m(\gamma)$; let $a_i\in\F_\ell$ be distinct and $g_1\neq g_2$.  Then Table \ref{tabreg} is a complete classification of relevant regular semisimple conjugacy class shapes.  
(In each of these cases, all $e_j=1$ and so we omit the trivial partition data from RCF.)


\begin{table}[h]
\caption{\label{tabreg}Regular semisimple conjugacy class shapes}
\[\begin{array}{||l||c|ll||}
\hline\hline
\text{Class shape} & f_\gamma(T) & m(\gamma) & \\ 
\hline 
\split & \prod_{j=1}^4 (T-a_j) & m=a_1a_3=a_2a_4 &  \\ 
\dqs & g_1(T)g_2(T) & m=g_j(0) & \text{(symplectic)}  \\ 
\dqi & g_1(T)g_2(T) & m\neq g_j(0), m^2=g_1(0)g_2(0) & \text{(isotropic)}  \\ 
\quartic & g(T) & m^2=g(0) & \\\hline \hline
\end{array}\]
\end{table}

\noindent\textbf{Case 2: Non-semisimple elements} 

As stated above, if $f_\gamma(T)$ is not squarefree then $\gamma$ is
cyclic if and only if all of the associated partitions are maximal
(consist of a single part).  In fact, such a conjugacy class is
determined by a signed partition -- the sign corresponds to a choice
of coset in $\ff_\ell\units/(\ff_\ell\units)^2$ as discussed above --  and Table \ref{tabnonreg} completes our list of relevant cyclic conjugacy class shapes.  (As in Table \ref{tabreg}, the $a_i\in\F_\ell$ are distinct.)


\begin{table}[h]
\caption{\label{tabnonreg}Non-semisimple cyclic conjugacy class shapes}
\[\begin{array}{||l||c|ll|c||}
\hline\hline
\text{Class shape} & f_\gamma(T) & m(\gamma) & 
& \text{Partition} \\
\hline 
\qrl & (T-a)^4 & m=a^2 & & [4] \\ 
\drls & (T-a)^2(T+a)^2 & m=a^2 & \text{(symplectic)} & \set{[2],[2]}\pm \\ 
\drli & (T-a_1)^2(T-a_2)^2 & m=a_1a_2 & \text{(isotropic)} & [2] \\  
\textbf{RQ-1} & [g(T)]^2 & m=g(0) &  & [2]  \\ 
\textbf{RQ-2} & (T^2-m)^2  & m \neq \Box &  & [2]\pm \\ \hline\hline

\end{array}\]
\end{table}

\noindent (For the conjugacy class shape \drli, the partition [2] corresponds to the factor $(T-a_1)(T-a_2)$ in $f_\gamma(T)$, and thus to a subspace with characteristic polynomial $(T-a_1)(T-a_2)$.)

Note that, for a fixed characteristic polynomial $f$ of shape \drls or
\textbf{RQ-2}, the set of cyclic elements with characteristic polynomial $f$
forms {\em two} conjugacy classes.  For example, for a nonsquare $x\in\F_\ell$, 
$$\gamma_1=\begin{pmatrix}
a&&1&\\
&-a&&1\\
&&a&\\
&&&-a
\end{pmatrix} 
\ \ \text{ and } \ \ 
\gamma_2=\begin{pmatrix}
a&&1&\\
&-a&&x\\
&&a&\\
&&&-a
\end{pmatrix}$$
are both elements of shape \drls.  (Verification of this fact is discussed in the proof of Lemma \ref{GSpnonregularcent}.) The matrix 
$$Z=\begin{pmatrix}
z_1&&z_3&\\
&z_2&&z_4\\
&&z_1&\\
&&&z_2x
\end{pmatrix}$$
conjugates $\gamma_1$ to $\gamma_2$ over $\GL_4(\F_\ell)$, but is an element of $\GSp_4(\F_\ell)$ if and only if $z_1^2=z_2^2 x$. Since $x$ is nonsquare, $\gamma_1$ and $\gamma_2$ are not conjugate in $\GSp_4(\F_\ell)$, although they {\em are} conjugate in $\GSp_4(\F_{\ell^2})$.   (A similar argument shows that we also have two classes of shape \textbf{RQ-2}.)

\subsection{Centralizer orders}

We determine the size of each of the conjugacy classes $\cclass(\gamma)$ listed in Tables \ref{tabreg} and \ref{tabnonreg} by computing the order of the centralizer of the representative $\gamma$.  Let $\cent_{\GSp_4(\F_\ell)}(\gamma)$ denote the centralizer of $\gamma$ in $\GSp_4(\F_\ell)$.

A representative of a regular semisimple conjugacy class is an element of a unique maximal torus of $\GSp_4(\F_\ell)$, and the centralizer of such a $\gamma$ is that maximal torus \cite{carter}.  We use the structure of these maximal tori to compute the sizes of the centralizers of the regular semisimple class shapes.

\begin{lemma} \label{GSptoruscent}
Let $\cclass(\gamma)$ have one of the conjugacy class shapes listed in Table \ref{tabreg}.  Then
$$\#\cent_{\GSp_4(\F_\ell)} (\gamma) = \begin{cases}
(\ell-1)^3 & \textup{ if $\cclass(\gamma)$ is \split}, \\
(\ell+1)^2(\ell-1) & \textup{ if  $\cclass(\gamma)$ is \dqs}, \\
(\ell+1)(\ell-1)^2 & \textup{ if  $\cclass(\gamma)$ is \dqi}, \\
(\ell^2+1)(\ell-1) & \textup{ if $\cclass(\gamma)$ is \quartic}.
\end{cases}$$
\end{lemma}

\begin{proof}
In each case, we determine the size of the appropriate torus.  For example, if $\cclass(\gamma)$ is of shape \quartic, the polynomial $f_\gamma(T)$ has roots $t, t^\ell, t^{\ell^2}$, and $t^{\ell^3}$ in $\F_{\ell^4}^\times$ in one orbit under the action of Galois.  Two pairs of roots have product $m(\gamma)\in \F_\ell^\times$ since $\gamma\in\GSp_4(\F_\ell)$.  
The element $t t^{\ell}$ cannot lie in $\F_\ell^\times$ when $t\in\F_{\ell^4}\smallsetminus \F_{\ell^2}$, thus $t t^{\ell^2}=m(\gamma)$.  The map $t \mapsto t t^{\ell^2}$ is the norm map of $\F_{\ell^4}$ over $\F_{\ell^2}$.  There are $\frac{\ell^4 -1}{\ell^2 -1} \cdot (\ell-1) = (\ell^2 +1)(\ell-1)$ elements of $\F_{\ell^4}^\times$ whose $\F_{\ell^2}$-norm lies in $\F_\ell^\times$, which is  the size of the torus and thus the centralizer.  The other centralizer orders are computed analogously.
\end{proof}

Determining the centralizer orders of the non-semisimple class shapes requires more effort.

\begin{lemma} \label{GSpnonregularcent}
Suppose $\cclass(\gamma)$ has one of the conjugacy class shapes listed in Table \ref{tabnonreg}.  Then 
$$\#\cent_{\GSp_4(\F_\ell)} (\gamma) = \begin{cases}
\ell^2(\ell-1) & \textup{ if $\cclass(\gamma)$ is \qrl}, \\
2\ell^2(\ell-1) & \textup{ if $\cclass(\gamma)$ is \drls}, \\
\ell(\ell-1)^2 & \textup{ if $\cclass(\gamma)$ is \drli}, \\
\ell(\ell^2-1) & \textup{ if $\cclass(\gamma)$ is \textbf{RQ-1}}, \\
2\ell^2(\ell-1) & \textup{ if $\cclass(\gamma)$ is \textbf{RQ-2}}.
\end{cases}$$
\end{lemma}

\begin{proof}
For each conjugacy class shape, find an explicit cyclic representative $\gamma\in\GSp_4(\F_\ell)$ such that $f_\gamma$ and $m(\gamma)$ are as given in Table \ref{tabnonreg}.  Then find a generic member $C$ of the centralizer of $\gamma$ and use it to find the size of $\cent_{\GSp_4(\F_\ell)} (\gamma) $.  (For the \drls and \textbf{RQ-2} shapes, two distinct non-conjugate representatives are needed for the $+$ and $-$ classes.)

%

As an example, reconsider our previous example where $\cclass(\gamma)$ is \drls.   It is easy to verify that the representatives $\gamma_1$ and $\gamma_2$ given earlier are cyclic elements of $\GSp_4(\F_\ell)$ with characteristic polynomial $(T-a)^2(T+a)^2$ and  $m(\gamma)=a^2$.   The matrix
$$C=\begin{pmatrix}
c_1&&c_3&\\
&c_2&&c_4\\
&&c_1&\\
&&&c_2
\end{pmatrix}$$
centralizes both $\gamma_1$ and $\gamma_2$ with the conditions that $c_1\in\F_\ell^\times$, $c_2=\pm c_1$, 
and $c_3,c_4\in\F_\ell$.  Then each class has a centralizer of order $2\ell^2(\ell-1)$.

The rest of the computations are similar and are omitted here.
\end{proof}

\section{Local factors for $f$}
Given $f$, we will define terms $\nu_\ell(f)$ for each finite prime of
$\ell$, as well as an archimedean term $\nu_\infty(f)$.
For finite primes $\ell\not = p$, we let $\nu_\ell(f)$ be the
probability that a random element of $\gsp_4(\F_\ell)^{(q)}$ has
characteristic polynomial $f$, and compare this probability to the
corresponding probability for a ``typical'' polynomial.  (This is a
higher-dimensional analogue of Gekeler's philosophy and concomitant
definition for elliptic curves \cite[Sec.\ 3]{gekeler03}; see Section
\ref{subsecgekeler} for a brief discussion of why, in the presence of
hypothesis \ref{enmax}, it suffices to work with a simpler definition
than that of \cite{gekeler03}.)
The definitions of $\nu_p(f)$ and
$\nu_\infty(f)$ are more intricate, but guided by a similar
philosophy.

\subsection{$\nu_\ell(f)$}

First, suppose $\ell\not = p$ is a finite rational prime. The
Frobenius endomorphism $\varpi_{X/\ff_q}$ of a principally polarized  abelian variety
$X/\ff_q$ acts an automorphism of $X_\ell$.  The polarization induces
a symplectic pairing on $X_\ell$; $\varpi_{X/\ff_q}$ scales this
pairing by a factor of $q$, and we may think of $\varpi_{X/\ff_q}$ as
an element of $\gsp(X_\ell)^{(q)}
\iso \gsp_4(\ff_\ell)^{(q)}$ (recall the notation surrounding
\eqref{diagsp}).

(Briefly) setting aside abelian varieties, there are 
$\ell^2$ polynomials which
occur as characteristic polynomials of elements of
$\gsp_4(\ff_\ell)^{(q)}$.    The average frequency (over all such polynomials)
 with which a given polynomial occurs as the characteristic
polynomial of an element of $\gsp_4(\ff_\ell)^{(q)}$ is
$\#\gsp_4(\ff_\ell)^{(q)}/\ell^2$. 

Consequently, at least for $\ell$ unramified in $K_f$,  we measure the departure of
the frequency of occurrence of $f$ from the average such frequency by
by
\begin{equation}
\label{eqdefnuell}
\nu_\ell(f) =
\frac{\#\st{ \gamma \in \gsp_4(\ff_\ell)^{(q)} : f_\gamma \equiv f \bmod
    \ell}}{\#\gsp_4(\ff_\ell)^{(q)}/\ell^2}.
\end{equation}
(An extension of this definition to all $\ell\not = p$ is given below
in \eqref{eqredefnuell}.)

\subsection{$\nu_p(f)$}

By way of motivation, suppose that $X/\ff_q$ is an ordinary abelian
surface, with characteristic polynomial of Frobenius $f_{X/\ff_q}(T) =
T^4 - a_X T^3+b_XT^2 - qa_XT + q^2$.  Since $X$ is ordinary (and
$\ff_q$ is perfect), there is
a canonical decomposition $X[p] \iso X[p]^\et \oplus X[p]^\tor$ of the
$p$-torsion group scheme into \'etale and toric components.  Note that
$X_p := X[p](\bar\ff_q)$ is actually $X[p]^\et(\bar\ff_q) \iso
(\integ/p)^2$.  The $\ff_q$-rational structure of $X[p]^\et$ is
captured by the action of the $q$-power Frobenius on $X_p$.  In fact,
$\varpi_{X/\ff_q}$ acts invertibly on $X_p$, with characteristic
polynomial $g_{X/\ff_q}(T) := T^2 - a_XT+b_X \bmod p$.

Now, $X[p]^\tor$ is connected, and specifically
$X[p]^\tor(\bar\ff_q)$ is a single point; but its Cartier dual is
\'etale.  In particular $(X[p]^\tor)^*(\bar\ff_q) \iso (\integ/p)^2$,
and the action of Frobenius on this Galois module (again) has
characteristic polynomial $g_{X/\ff_q}(T)$.  

Finally, recall that the Frobenius operator must preserve the canonical
decomposition of $X[p]$ into its \'etale and toric parts.

Because of these considerations, we set
\[
\nu_p(f) = \frac{\#\st{\gamma \in \gsp_4(\ff_p)^{(b^2)} :
    f_\gamma \equiv (T^2-aT+b)^2 \bmod p\text{ and
    }\gamma\text{ semisimple}}}
{\#\gsp_4(\ff_p)^{(b^2)}/p^2}.
\]

\subsection{$\nu_\infty(f)$}

It remains to define an archimedean term; our choice comes from the
Sato-Tate measure, which (conjecturally) explains the distribution of
Frobenius elements of abelian surfaces.

Recall that semisimple conjugacy classes in the compact group ${\rm USp}_4$
are parametrized by (``Frobenius angles'') $0 \le \theta_1 \le
\theta_2\le\pi$.  The Sato-Tate measure on the space of Frobenius angles is
simply the pushforward of Haar measure.  Explicitly, the Weyl
integration formula \cite[p218,7.8B]{weyl} shows that this measure is
\[
\mu_\satotate(\theta_1,\theta_2) =
\frac{16}{\pi^2}(\cos(\theta_2)-\cos(\theta_1))^2 \sin^2(\theta_1)
\sin^2(\theta_2)\, d\theta_1\, d\theta_2.
\]
Once $q$ is fixed, a pair of angles $\st{\theta_1,\theta_2}$ gives
rise to a $q$-Weil polynomial
\[
\prod_{j = 1,2} (T - \sqrt q\exp(i\theta_j))(T - \sqrt q
\exp(-i\theta_j));
\]
the induced measure on the space of $q$-Weil polynomials is
\[
\mu_\satotate(a,b) = \oneover{4q^3\pi^2}
\sqrt{(a^2-4b+8q)(b^2+4bq+4q^2-4a^2q)}\, da\, db.
\]
Note that, since there are approximately $q^{\dim \stk A_2} = q^3$
principally polarized abelian surfaces over $\ff_q$,
abelian varieties, $q^3 \mu_\satotate(a,b)$ is a sort of archimedean
prediction for $\#\stk A_2(\ff_q;f)$.  Guided by this and the
calculations of Lemma \ref{lemconductorq}, 
we set
\begin{equation}
\label{eqdefnuinfty}
\nu_\infty(f) = \oneover{\cond(f) \, 4\pi^2}
\sqrt{\abs{\frac{\Delta_f}{\Delta_{f^+}}}}.
\end{equation}

\section{The shape of Frobenius}
\label{secfrob}

Fix a $q$-Weil polynomial satisfying conditions \ref{enord}-\ref{enmax}.  To ease notation slightly, we will write $K$ for $K_f$ and, given assumption \ref{enmax}, write $\calo_K$ for $\calo_f$.  Suppose $X/\ff_q$ is a principally polarized abelian variety such that $\calo_K \subseteq \End(X)$; we choose the polarization so that the Rosati involution on $\End(X)$ induces complex conjugation on $\calo_K$.  On $\ell$-torsion, the principal polarization induces a symplectic pairing on $X_\ell$; complex conjugation on $\calo_K\tensor \ff_\ell$ is adjoint with respect to this pairing; and we obtain $\rho_\ell(\varpi_f) \in \gsp(X_\ell)$.  Our goal in the present section is to relate the shape of $\rho_\ell(\varpi_f)$ (in the sense of Section \ref{secclassshapes}) to the structure of $f(T) \bmod \ell$.

Of course, all of this can be formulated without recourse to abelian varieties.  Let $\kappa(\ell) = \calo_K\tensor\ff_\ell$; it is a four-dimensional vector space over $\ff_\ell$.  Choose a symplectic pairing $\ang{\cdot,\cdot}$ on $\kappa(\ell)$ for which complex conjugation on $\calo_K\tensor\ff_\ell$ is the adjoint with respect to $\ang{\cdot,\cdot}$.  (If $\ell\nmid \Delta_K$, one may explicitly construct such a pairing as follows.  Choose $\alpha \in \calo_K$ relatively prime to $\ell$ such that $\bar\alpha = -\alpha$. Then the reduction modulo $\ell$ of the pairing
\begin{diagram}
\calo_K \cross \calo_K & \rto  & \integ \\
(x,y) & \rmto & {\rm tr}_{K/\rat}(\alpha x\bar y)
\end{diagram}
is a suitable form.)  Such a form is canonically defined up to scaling, and in particular its group of symplectic similitudes is independent of the choice of form.

Then $\varpi_f$ acts on $\kappa(\ell)$.  Let $\gamma_\ell$ be the image of $\varpi_f$ in $\gsp(\kappa(\ell))$; our goal is to use the splitting behavior of $f(T)\bmod \ell$ to compute the {\em cyclic shape} of $\gamma_\ell$, i.e., the shape of any cyclic element whose semisimplification is conjugate to $\gamma_\ell$.

In fact, we define
\begin{equation}
\label{eqredefnuell}
\nu_\ell(f) = \frac{
\#\st{
\gamma \in \gsp_{4}(\ff_\ell): \gamma\text{ is cyclic, with semisimplification } \gamma_\ell}}
{\#\gsp_4(\ff_\ell)^{(q)}/\ell^2}.
\end{equation}

\begin{lemma}
If $\ell\nmid p\Delta_K$, then definitions \eqref{eqdefnuell} and \eqref{eqredefnuell} coincide.
\end{lemma}

\begin{proof}
If $\ell\nmid p\Delta_K$, then $\ell\nmid \Delta_f$.  Therefore, $f(T)\bmod \ell$ has distinct roots, and $\gamma_\ell$ is regular semisimple.  The classification in Table \ref{tabreg} shows that if $f(T)\bmod \ell$ is either irreducible or a product of linear factors, then any element with characteristic polynomial $f(T)\bmod \ell$ is actually conjugate to $\gamma_\ell$.  If $f(T)\bmod \ell$ is a product of distinct irreducible quadratic polynomials, then the possible shapes of $\gamma_\ell$ are distinguished by their multiplier; but one knows that the multiplier of $\gamma_\ell$ is $q$.  The claim now follows once one recalls that a regular semisimple element is cyclic.
\end{proof}

Note that, tautologically, the characteristic polynomial of $\gamma_\ell$ is exactly the reduction of $f(T)$.  If, for instance, $f(T) \bmod \ell$ is irreducible, then a moment's reflection (or a glance at Tables \ref{tabreg} and \ref{tabnonreg}) reveals that $\gamma_\ell$ is \quartic.

However, it sometimes happens (e.g., with \dqs and \dqi) that the factorization pattern of $f$ alone does not determine the shape of $\gamma$.

Since $\kappa(\ell) = \calo_K/\ell \iso \ff_\ell[T]/f(T)$ (Corollary \ref{lemalmostmaximal}), the factorization of $f(T) \bmod \ell$ is precisely determined by the splitting of $\ell$ in $\calo_K$.  We have 
\[
\kappa(\ell) \iso \bigoplus_{\lambda | \ell}  \kappa(\ell)_\lambda,
\]
where $\lambda$ ranges over all primes of $K$ which lie over $\ell$. 
 (In fact, the dimension of $\kappa(\ell)_\lambda$ over the residue field $\calo_K/\lambda$ is $e(\lambda/\ell)$, the ramification index of $\lambda$.)


For the sequel, it is worth singling out the following immediate observation:

\begin{lemma}
\label{lemlambdabarlambda}
The symplectic pairing induces a perfect duality between $\kappa(\ell)_\lambda$  and $\kappa(\ell)_{\bar\lambda}$.
\end{lemma}

\begin{proof} 
We have chosen $\ang{\cdot,\cdot}$ such that the involution induced by complex conjugation is the adjoint with respect to $\ang{\cdot,\cdot}$.
\end{proof}

Consider a finite Galois extension  $L/\Q$ with $\gal(L/\Q)=G$.  If $\ell$ is a rational prime and $\lambda$ is a prime of $\O_L$ lying over $\ell$, the decomposition group and inertia group of $\lambda$ are, respectively,
\begin{align*}
D(\lambda/\ell) &= \st{\sigma \in G: \sigma(\lambda) = \lambda} \\
I(\lambda/\ell) &= \st{ \sigma \in G: \forall \beta \in \calo_L : \sigma(\beta) \equiv \beta \bmod \lambda}.
\end{align*}
Then $I(\lambda/\ell)$ is normal in $D(\lambda/\ell)$.  In fact, we will only use these notions for the abelian extension $K/\rat$, and thus the inertia and decomposition groups depend only on $\ell$, and not on the choice of $\lambda$.  Hence we write $I(\ell)$ and $D(\ell)$ for $I(\lambda/\ell)$ and $D(\lambda/\ell)$.

Let $f(T) \equiv \prod_{1 \le j \le r} g_j(T)^{e_j} \bmod \ell$ be the factorization of $f(T) \bmod \ell$ into irreducible monic polynomials.  
Since $\calo_K/\ell \iso \ff_\ell[T]/f(T)$, there are $r$ primes, $\lambda_1, \cdots, \lambda_r$ of $K$ lying over $\ell$; $\calo_K/\lambda_i$ has degree $\deg g_i$ over $\ff_\ell$; and the ramification index of $\lambda_i$ is $e_i$.  Note that the quantities $\deg g_i$ and $e_i$ are independent of $i$ as $K/\rat$ is Galois.  (This is why we restricted to relevant conjugacy classes in Section \ref{secclassshapes}.)

%

Finally, if $\ell\nmid \Delta_f$ then there exists an element of $\Gal(K/\Q)$ which induces the canonical generator of $\gal(\kappa(\lambda)/\ff_\ell)$.  Let $\Frob_K(\ell)\in\Gal(K/\Q)$ be this element, called the \emph{Frobenius endomorphism} of $\lambda$ over $\ell$.  

\subsection{$K$ cyclic}
\label{subsecshapecyclic}

Suppose that $\gal(K/\rat)$ is cyclic, with generator $\sigma$.  Note that complex conjugation is given by $\cc = \sigma^2$. We classify the splitting behavior of rational primes $\ell$ in $K$ by enumerating the possibilities for $D(\ell)$ and $I(\ell)$.

\begin{lemma}
\label{lemcycliccorrespondence}
Suppose $f$ satisfies assumptions \ref{enord}-\ref{enmax} with cyclic Galois group generated by $\sigma$. Let $\ell\not = p$ be a rational prime.  The cyclic shape of $\gamma_\ell$ is determined by the decomposition and inertia groups $D(\ell)$ and $I(\ell)$ as in Table \ref{tabcycliccorrespondence}.
\end{lemma}

Thus, for instance, Lemma \ref{lemcycliccorrespondence} asserts that
if $D(\ell) = \ang{\sigma^2}$ and $I(\ell) = \ang 1$, then
$\gamma_\ell$ has cyclic shape \dqs.  Note that if $\gamma_\ell$ has
cyclic shape \textbf{RQ-2}, then there are two conjugacy classes with cyclic
shape $\gamma_\ell$.  Otherwise, the cyclic shape of $\gamma_\ell$
determines a unique conjugacy class.

\begin{proof}
  In Table \ref{tabcycliccorrespondence}, we have enumerated all
  possibilities for pairs of subgroups $I(\ell) \subseteq D(\ell)
  \subseteq \gal(K/\rat)$.  For each such pair, in the prime
  factorization of $f(T) \bmod \ell$, there are
  $r=\#\gal(K/\rat)/\#D(\ell)$ distinct irreducible factors. Each has
  degree $f=\#D(\ell)/\#I(\ell)$ and multiplicity $e=\#I(\ell)$ .  
When $I(\ell) \subseteq D(\ell)$ is either $\st 1 \subseteq 
\st 1$, $\st 1 \subset \ang\sigma$, or $\ang\sigma \subseteq
\ang\sigma$, this factorization pattern already determines the cyclic
shape of $\gamma_\ell$.

Of the remaining cases, we first consider those in which $D(\ell) =
\ang{\sigma^2}$.  Let $\lambda$ be one of the two primes of $K$ lying
over $\ell$.  Lemma \ref{lemlambdabarlambda} shows that
$\kappa(\ell)_\lambda$ is symplectic if and only if 
  complex conjugation stabilizes $\lambda$, i.e., if and only if $\cc
  = \sigma^2 \in D(\ell)$.  Since this happens in the two cases under
  consideration, the induced decomposition is symplectic and the cyclic
  shape of $\gamma_\ell$ is the {\bf S}
  variant.  

Finally, we analyze the situation in which $I(\ell) = \ang{\sigma^2}
\subset D(\ell) = \ang\sigma$; we must decide whether the cyclic shape
is \textbf{RQ-1} or \textbf{RQ-2}.  Let $\lambda$ be the prime of $K$ lying
over $\ell$.  Consider the Frobenius element
$\varpi_f$ as an element of $\calo_K$. Then $f(T) = \prod_{0 \le j \le
3}(T-\sigma^j(\varpi_f))$.  The ramification hypothesis implies that 
$\sigma^2(\varpi_f) \equiv \varpi_f \bmod \lambda$, and we have the
factorization
\begin{align*}
f(T) &\equiv g(T)^2 \bmod \lambda\\
\intertext{where}
g(T) & \equiv (T-\varpi_f)(T-\sigma(\varpi_f)) \bmod \lambda.
\end{align*}
By comparing constant terms, we find that
\begin{align*}
(\varpi_f \sigma(\varpi_f))^2 & \equiv q^2 \bmod \lambda\\
\intertext{and in particular} 
\varpi_f \sigma(\varpi_f) & \equiv \pm q \bmod \lambda.
\end{align*}
However, if we had $\varpi_f \sigma(\varpi_f) \equiv q \bmod \lambda$,
then we would know that $\sigma(\varpi_f)\equiv \sigma^2(\varpi_f)
\bmod \lambda$, which contradicts the hypothesis that $\sigma \not \in
I(\ell)$.  

Therefore, the constant term of the irreducible factor $g(T)$ is $-q \bmod
\lambda$, and the cyclic shape of $\gamma_\ell$ is \textbf{RQ-2}.
\end{proof}

\begin{table}[h]
\caption{\label{tabcycliccorrespondence}Prime factorizations and conjugacy class shapes for $K/\Q$ cyclic}
\centering
\begin{tabular}{||c|c|c|c|c||}
\hline\hline
$D(\ell)$ & $I(\ell)$ & $\Frob_K(\ell)$ & $(e,f,r)$ & {Class shape} \\ 
\hline
$\set{1}$ & $\set{1}$ & 1 & (1,1,4) & \split \\ 
$\gen{\sigma^2}$ & $\set{1}$ & $\sigma^2$ & (1,2,2) &\dqs  \\   
$\gen{\sigma^2}$ & $\gen{\sigma^2}$ & - & (2,1,2) & \drls   \\  
$\gen{\sigma}$ & $\set{1}$ & $\sigma$ \text{ or } $\sigma^3$ & (1,4,1) & \quartic \\ 
$\gen{\sigma}$ & $\gen{\sigma^2}$ & - & (2,2,1) & \textbf{RQ-2}  \\   
$\gen{\sigma}$ & $\gen{\sigma}$ & - & (4,1,1) & \qrl \\ \hline\hline
\end{tabular}
\label{cyclicfieldtable}
\end{table}

\subsection{$K$ biquadratic}
\label{subsecshapebiquadratic}

Suppose instead that $K=\Split(f)$ is biquadratic.  Then $K$ is the compositum of quadratic imaginary fields $K_1$ and $K_2$, and we have $\gal(K/\rat) \cong \gal(K_1/\Q) \oplus \gal(K_2/\Q) \iso \Z/2 \oplus \Z/2$; let $\tau_i$ generate $\gal(K/K_i)$.  Then complex conjugation is given by $\cc=\tau_1\tau_2$; its fixed field is the real quadratic subfield $K^+$.
%
%
We again classify the splitting behavior of primes $\ell$ by considering pairs $D(\ell)$ and $I(\ell)$.

\begin{lemma}
\label{lembiquadraticcorrespondence}
Suppose $f$ satisfies conditions \ref{enord}-\ref{enmax} with $K=K_f$ biquadratic.  Let $\ell\not = p$ be a rational prime.  The cyclic shape of $\gamma_\ell$ is determined by the decomposition and inertia groups $D(\ell)$ and $I(\ell)$ as in Table \ref{tabbiquadraticcorrespondence}.
\end{lemma}

As before, the cyclic shape of $\gamma_\ell$ determines a unique conjugacy class except when that shape is either \drls or \textbf{RQ-2}.  In these cases, there are two conjugacy classes with the cyclic shape given by $\gamma_\ell$.

\begin{proof}
Proceed as in the proof of Lemma \ref{lemcycliccorrespondence}.
Note that, by Lemma \ref{lemlambdabarlambda}, $\kappa(\ell)_\lambda$
is isotropic if and only if complex conjugation acts nontrivially on
$\lambda$, i.e., if and only if $\cc=\tau_1\tau_2 \notin D(\ell)$.  In
these cases the induced decomposition is isotropic and the cyclic
shape of $\gamma_\ell$ is the \textbf{I} variant. All ambiguous cases
where $r>1$ can be identified by the action of the pairing on the
induced decomposition.  

Now (without loss of generality) suppose that $I(\ell) = \ang{\tau_1} \subset D(\ell) =
\ang{\tau_1,\tau_2}$, and let $\lambda$ be the prime lying over $\ell$.  As in Lemma \ref{lemcycliccorrespondence}, we
recall that
\[
f(T) \equiv (T - \varpi_f)
(T-\tau_1(\varpi_f))(T-\tau_2(\varpi_f))(T-\tau_1\tau_2(\varpi_f))
\bmod \lambda.
\]
The assumption on ramification implies that $\varpi_f \equiv
\tau_1(\varpi_f)\bmod \lambda$, and thus that $\tau_2(\varpi_f) \equiv
\tau_1\tau_2(\varpi_f)\bmod \lambda$.  Therefore, $f(T)$ factors as
\[
f(T) \equiv ((T - \varpi_f)(T - \tau_2(\varpi_f)))^2 \bmod \lambda.
\]
Moreover, $\varpi_f \tau_2 (\varpi_f) \equiv \varpi_f
\tau_1\tau_2(\varpi_f) \equiv q\bmod \lambda$.  Therefore, $f(T)
\equiv g(T)^2 \bmod \lambda$ where $g(0) = q$, and the cyclic shape of
$\gamma_\ell$ is \textbf{RQ-1}.

The remaining cases follow in an analogous fashion.
\end{proof}

\begin{table}[h]
\caption{\label{tabbiquadraticcorrespondence}Prime factorizations and conjugacy class shapes for $K/\Q$ biquadratic}
\centering
\begin{tabular}{||c|c|c|c|c||}
\hline\hline
$D(\ell)$ & $I(\ell)$ & $\Frob_K(\ell)$ & $(e,f,r)$ & \text{Class shape} \\ 
\hline
$\set{1}$ & $\set{1}$ & 1 & (1,1,4) & \split \\ 
$\gen{\tau_i}$ & $\set{1}$ & $\tau_i$ & (1,2,2) & \dqi   \\   
$\gen{\tau_i}$ & $\gen{\tau_i}$ & -  & (2,1,2) & \drli   \\   
$\gen{\tau_1\tau_2}$ & $\set{1}$ & $\tau_1\tau_2$ & (1,2,2) & \dqs   \\   
$\gen{\tau_1\tau_2}$ & $\gen{\tau_1\tau_2}$ & - & (2,1,2) & \drls   \\   
$\gen{\tau_1,\tau_2}$ & $\gen{\tau_i}$ & - & (2,2,1) & \textbf{RQ-1} \\  
$\gen{\tau_1,\tau_2}$ & $\gen{\tau_1\tau_2}$ & - & (2,2,1) & \textbf{RQ-2} \\ 
\hline\hline
\end{tabular}
\label{biquadraticfieldtable}
\end{table}

\section{Local terms for $K$}

Let $\gal(K/\rat)^*$ be the character group of the Galois group of $K$. For $\chi \in \gal(K/\rat)^*$, let $K^\chi$ be the subfield of $K$ fixed by $\ker\chi$.  For a rational prime $\ell$, let $\chi(\ell) = \chi(\Frob_{K^\chi}(\ell))$ if $\ell$ is unramified in $K^\chi$, and let $\chi(\ell) = 0$ otherwise.

Since $K^+$ is a subextension of $K$, $\gal(K^+/\rat)^*$ is naturally a subgroup of $\gal(K/\rat)^*$, and we define
\begin{align}
\label{eqdefnuk}
\nu_\ell(K) & = \prod_{\chi \in S(K)} \oneover{1-\chi(\ell)/\ell}
\intertext{where }
S(K) &= \gal(K/\rat)^* \setcomp \gal(K^+/\rat)^*.
\end{align}


\subsection{$K$ cyclic}

As in Section \ref{subsecshapecyclic}, let $\gal(K/\rat) = \ang \sigma$.  Let $\chi$ be a
faithful character of $\gal(K/\rat)$, so that $\bar\chi := \chi \comp
\iota$ is the other faithful character of $\gal(K/\rat)$ and
$S(K) = \st{\chi,\bar\chi}$.

\begin{lemma}
\label{lemcycliccharval}
Let $\ell$ be a rational prime.  The multiset of values
$\st{\chi(\ell),\bar\chi(\ell)}$ is determined by the decomposition
and inertia groups $D(\ell)$ and $I(\ell)$ as in Table \ref{tabcyclicnuellk}.
\end{lemma}

\begin{proof}
This follows from the definition of $\st{\chi,\bar\chi}$ and the
calculation of Frobenius elements in Lemma \ref{lemcycliccorrespondence}. In particular, 
$\st{\chi(\ell),\bar\chi(\ell)}$ depends only on the order of
$D(\ell)$ and $I(\ell)$, and not on their canonical generators; and is
also independent of the choice of generator of $\gal(K/\rat)^*$.
\end{proof}

\begin{table}[h]
\caption{\label{tabcyclicnuellk}Values of imaginary characters on
  Frobenius elements for $K/\rat$ cyclic.}
\[
\begin{array}{||c|c|c|c||}
\hline\hline
D(\ell) & I(\ell) &\st{\chi(\ell),\bar\chi(\ell)} &\text{Class shape}\\ \hline
\st 1 & \st 1 & \st{1,1} &\split \\
\ang{\sigma^2} & \st 1 & \st{-1,-1} & \dqs\\
\ang{\sigma^2} & \ang{\sigma^2} & \st{0,0}&\drls \\
\ang \sigma & \st 1 & \st{-i,i} &\quartic\\
\ang \sigma & \ang{\sigma^2} & \st{0,0} &\textbf{RQ-2}\\
\ang \sigma & \ang\sigma & \st{0,0}&\text{\qrl}\\ \hline\hline
\end{array}
\]
\end{table}

\subsection{$K$ biquadratic}

%

As in Section \ref{subsecshapebiquadratic}, let  $\gal(K/\Q) = \ang{\tau_1,\tau_2}$.  Denote the quadratic imaginary subfields of $K$ by  $K_1$ and $K_2$ and let $\gal(K/K_i) = \ang{\tau_i}$.  For $i \in \st{1,2}$, define
the character
\[
\phi_i(\tau_j) := \begin{cases} -1 & i=j, \\
1 & i\not = j.
\end{cases}
\]
Then $S(K)  = \st{\phi_1,\phi_2}$.

\begin{lemma}
\label{lembiquadraticcharval}
Let $\ell$ be a rational prime.  The multiset of values
$\st{\phi_1(\ell),\phi_2(\ell)}$ is determined by the decomposition
and inertia groups $D(\ell)$ and $I(\ell)$ as in Table \ref{tabbiquadraticnuellk}.
\end{lemma}

\begin{table}[h]
\caption{\label{tabbiquadraticnuellk}Values of imaginary characters on
  Frobenius elements for $K/\rat$ biquadratic.}
\[
\begin{array}{||c|c|c|c||}
\hline\hline
D(\ell) & I(\ell) &\st{\phi_1(\ell),\phi_2(\ell)} &\text{Class shape}\\ \hline
\st 1 & \st 1 & \st{1,1} &\split\\
\ang{\tau_i} & \set{1} & \st{-1,1} & \dqi\\
\ang{\tau_i} & \ang{\tau_i} &  \st{0,1}& \drli\\
\ang{\tau_1\tau_2} & \set 1 & \st{-1,-1} &\dqs\\
\ang{\tau_1\tau_2} & \ang{\tau_1\tau_2} & \st{0,0}& \drls\\
\ang{\tau_1,\tau_2} & \ang{\tau_i} & \st{0,-1} &\textbf{RQ-1} \\
\ang{\tau_1,\tau_2} & \ang{\tau_1\tau_2} & \st{0,0} &\textbf{RQ-2} \\\hline\hline
\end{array}
\]
\end{table}

\subsection{Matching}

In this section, we show that each of the local factors na\"ively assigned to $f$ matches a factor intrinsic to the splitting field $K = K_f$.

\begin{proposition}
\label{propmatchell}
If $\ell\not = p$, then
\[
\nu_\ell(f) = \nu_\ell(K).
\]
\end{proposition}

\begin{proof}
Let $\gamma_\ell$ be as in Section \ref{secfrob}.  We first assume that the cyclic shape
of $\gamma_\ell$ determines a unique conjugacy class in
$\gsp_4(\ff_\ell)$, and then indicate what must be changed to
accommodate the remaining cases. 

Thus, let $\gamma$ be any cyclic element
whose semisimplification is $\gamma_\ell$, and assume the shape of
$\gamma$ is neither \drls nor \textbf{RQ-2}. 
By Lemmas \ref{lemcycliccharval} and \ref{lembiquadraticcharval}, the
set of character values
\[
\st{\chi(\ell) : \chi \in S(K)}
\]
depends only on the shape of $\gamma$; and tautologically, the size of
the conjugacy class $\calc(\gamma)$ only depends on the shape of
$\gamma$, too.   

On one hand, by \eqref{eqredefnuell} we have 
\begin{align*}
\nu_\ell(f) &= \frac{\#\calc(\gamma)}{\#\GSp_4(\ff_\ell)^{(q)}/\ell^2} \\
&=
\frac{\#\gsp_4(\ff_\ell)/\#\cent(\gamma)}{\#\GSp_4(\ff_\ell)^{(q)}/\ell^2}
\\
&= \frac{\ell^2(\ell-1)}{\#\cent(\gamma)}.
\end{align*}
Lemmas \ref{GSptoruscent} and \ref{GSpnonregularcent}
supply column 2 of Table \ref{tabmatching}, and applying this simple calculation provides column 3.

On the other hand, recall that \eqref{eqdefnuk} gives
\[
\nu_\ell(K) = \prod_{\chi \in S(K)}\oneover{1-\chi(\ell)/\ell}.
\]
Lemmas \ref{lemcycliccharval} and \ref{lembiquadraticcharval}
provide column 4 of Table \ref{tabmatching}, and we compute column 5 using \eqref{eqdefnuk}.

If $\gamma_\ell$ has cyclic shape of type \drls or \textbf{RQ-2}, then there
are two cyclic conjugacy classes with semisimplification
$\gamma_\ell$.  For a representative $\gamma$ of each class,
$\#\calc(\gamma)/(\#\GSp_4(\ff_\ell)^{(q)}/\ell^2) = \half 1$, and thus
$\nu_\ell(f) = \half 1 + \half 1$.

As columns 3 and 5 are equal, the theorem is proven.
\end{proof}

\begin{table}[h]
\caption{\label{tabmatching}$\nu_\ell(f)$ and $\nu_\ell(K)$}
\[
\begin{array}{||c|c|c|c|c||}
\hline\hline
\text{Class shape} & \#\cent(\gamma) & \nu_\ell(f) & \st{\chi(\ell): \chi
  \in S(K)} & \nu_\ell(K) \\ \hline
\split & (\ell-1)^3 & \frac{\ell^2}{(\ell-1)^2} & \st{1,1} &
\frac{\ell}{\ell-1}\cdot \frac{\ell}{\ell-1} \\ 
\dqs &(\ell+1)^2(\ell-1) &
\frac{\ell^2}{(\ell+1)^2}&\st{-1,-1}&
\frac{\ell}{\ell+1}\cdot\frac{\ell}{\ell+1}\\ 
\dqi & (\ell+1)(\ell-1)^2& 
\frac{\ell^2}{\ell^2-1}&\st{-1,1}&\frac{\ell}{\ell+1}\cdot\frac{\ell}{\ell-1}\\ 
\quartic& (\ell^2+1)(\ell-1)&
 \frac{\ell^2}{\ell^2+1}&\st{-i,i}&
 \frac{\ell}{\ell-i}\cdot\frac{\ell}{\ell+i}\\ 
\qrl& \ell^2(\ell-1) &1&\st{0,0} & 1 \cdot 1 \\
\drls& 2\ell^2(\ell-1)& 1&\st{0,0} & 1 \cdot 1 \\ 
\drli & \ell(\ell-1)^2&
\frac{\ell}{\ell-1} &\st{0,1} &\frac{\ell}{\ell-1} \cdot 1\\ 
\textbf{RQ-1} & \ell(\ell^2-1) & 
\frac{\ell}{\ell+1} &\st{-1,0}& \frac{\ell}{\ell+1} \cdot 1\\ 
\textbf{RQ-2} & 2\ell^2(\ell-1)& 1 & \st{0,0} & 1 \cdot 1 \\ \hline\hline
\end{array}
\]
\end{table}

Similarly:

\begin{lemma}
\label{lemmatchp}
We have
\[
\nu_p(f) = \nu_p(K).
\]
\end{lemma}

\begin{proof}
Since we have assumed $p$ unramified in $K$ (\ref{engal}), $g(T) := T^2-aT+b$
is not a square.  For convenience, we recall the definition
\[
\nu_p(f) = \frac{\#\st{\gamma \in \gsp_4(\ff_p)^{(b^2)} :
    f_\gamma(T) \equiv g(T)^2\bmod p\text{ and
    }\gamma\text{ semisimple}}}
{\#\gsp_4(\ff_p)^{(b^2)}/p^2}.
\]

First suppose that $K/\rat$ is cyclic.  Then $p$ splits completely in
$K$ (e.g., \cite[Table 3]{gorenlauter12}), and $g(T)$ factors (in $\ff_p$).
The set of {\em semisimple} elements with characteristic polynomial
$g(T)^2$ has the same cardinality as a conjugacy class of type
$\split$.  From (the first line of) Table \ref{tabmatching}, we see
that $\nu_p(f) = \nu_p(K)$.

Now instead suppose that $K/\rat$ is biquadratic.  Then either $p$
splits completely in $K$, or $p$ splits in exactly one of the $K_i$
(\cite[Table 4]{gorenlauter12}).  The former case has already been
addressed.  For the latter case, the set of semisimple elements with
characteristic polynomial $g(T)^2$ has the same cardinality as a
conjugacy class of type $\dqi$.  Again we conclude from Table
\ref{tabmatching} that $\nu_p(f) = \nu_p(K)$.
\end{proof}

Finally, we compute:

\begin{lemma}
\label{lemmatchinfinity}
We have
\[
\nu_\infty(f) = \oneover{4\pi^2}\sqrt{\abs{\frac{\Delta_K}{\Delta_{K^+}}}}.
\]
\end{lemma}

\begin{proof}
From Lemma \ref{lemconductorq} we have $\cond(f)=q$ and  $\Delta_f =q^2 \Delta_K$.  Additionally, Lemma \ref{lemmaximalrealorder} implies that $\Delta_{f^+}=\Delta_{K^+}$.  Then \eqref{eqdefnuinfty} gives the result.
\end{proof}

\subsection{Comparison with \cite{gekeler03}}
\label{subsecgekeler}

As we mentioned in the introduction, the present work is inspired by
Gekeler's work \cite{gekeler03} with ordinary isogeny classes of
elliptic curves over $\ff_p$.  He starts with an ordinary quadratic
$p$-Weil polynomial $g(T)$; defines
\begin{equation}
\label{eqgekeler}
\nu_\ell^{\operatorname{G}}(g) =\lim_{r \ra\infty}\frac{\#\st{\gamma \in \GL_2(\integ/\ell^r)^{(q)}:
    f_\gamma \equiv g \bmod
    \ell^r}}{\#\GL_2(\integ/\ell^r)^{(q)}/\ell^r};
\end{equation}
and among other results shows \cite[Cor.\ 4.8]{gekeler03} that if $\ell^2\nmid \Delta_g$,
then $\nu_\ell^{\operatorname{G}}(g) = 1/(1-\chi(\ell)/\ell)$, where $\chi$ is the
quadratic character of the splitting field of $g$.

If $\ell \nmid \Delta_g$, then the centralizer of an element with
characteristic polynomial $g$ is smooth over $\integ_\ell$, and thus
setting $r=1$ in the right-hand side of \eqref{eqgekeler} already
calculates the limiting value.  If $\ell$, but not its square, divides
$\Delta_g$, then \eqref{eqgekeler} does not stabilize at $r=1$.
However, if instead we ask for the proportion of {\em cyclic } elements of
$\gl_2(\ff_\ell)$ with characteristic polynomial, then we again
have a finite expression which computes
$\nu_\ell^{\operatorname{G}}(g)$.

Returning to the context of abelian surfaces, the same ``smoothness of
centralizers'' argument shows that if $\ell \nmid \Delta_f$, then the proportion in
\eqref{eqdefnuell} calculates the $\gsp_4$-analogue of the limit in
\eqref{eqgekeler}.  Condition \ref{enmax}, which corresponds to the
local condition $\ell^2\nmid \Delta_g$, is why passing to cyclic shape
in \eqref{eqredefnuell} again allows us to compute in $\ff_\ell$,
as opposed to $\integ_\ell$.

\section{Main result}

In the following, we will have several occasions to consider conditionally convergent infinite products.  For a sequence of numbers $\st{a_\ell}$ indexed by finite primes, let
\begin{equation}
\label{eqcondcvg}
\prod_\ell a_\ell = \lim_{X \ra \infty}\prod_{\ell < X}a_\ell.
\end{equation}
With this convention, we have $(\prod_\ell a_\ell) \cdot (\prod_\ell b_\ell) = \prod_\ell (a_\ell b_\ell)$.

For a number field $L$, let $h(L)$, $\omega_L$ and $R_L$ denote, respectively, the class number, number of roots of unity, and regulator of $L$.

\begin{theorem}
\label{thmain}
Let $f$ be a degree $4$ $q$-Weil polynomial which is ordinary,
principally polarizable, Galois and maximal.  Let $K_f$ be the
splitting field of $f$, and let $K^+_f$ be its maximal totally real
subfield.  Then
\begin{equation}
\label{eqmain}
\nu_\infty(f) \prod_{\ell} \nu_\ell(f) = \oneover{\omega_K}\frac{h(K_f)}{h(K_f^+)}.
\end{equation}
\end{theorem}

\begin{proof}
We write $K$ and $K^+$ for $K_f$ and $K^+_f$.
By the analytic class number formula, the ratio of class numbers on the right-hand side of 
\eqref{eqmain} is
\[
\frac{h(K)}{h(K^+)} = \lim_{s \ra 1} \frac{(s-1)\zeta_K(s)}{(s-1)\zeta_{K^+}(s)}
\frac{\sqrt{\abs{\Delta_K}}  2^2\omega_K R_{K^+}}
{\sqrt{\abs{\Delta_{K^+}} (2\pi)^2 \omega_{K^+} R_K}}.
\]
For a finite
abelian extension $L/\rat$, we have
\[
\lim_{s \ra 1}(s-1)\zeta_L(s) = \prod_{\chi \in \gal(L/\rat)^*\setcomp \id}
L(1,\chi),
\]
where the product is over nontrivial characters of $\gal(L/\rat)$ and, as in \eqref{eqcondcvg},
we interpret $L(1,\chi)$ as the conditionally convergent product
\[
L(1,\chi) = \lim_{X \ra \infty} \prod_{\ell <
  X}\oneover{1-\chi(\ell)/\ell}.
\]
With our convention on conditionally convergent products, 
\[
\prod_{\chi \in \gal(L/\rat)^*\setcomp \id} L(1,\chi) = 
\prod_{\ell}\left(\prod_{\chi \in \gal(L/\rat)^*\setcomp \id}
\oneover{1-\chi(\ell)/\ell}\right).
\]
By hypothesis $K$ and $K^+$ are abelian, as they are Galois over $\rat$ of degrees $4$ and $2$, respectively.  By definition \eqref{eqdefnuk}, for each $\ell$, 
\[
\frac{\begin{displaystyle} \prod_{\chi \in \gal(K/\rat)^*\setcomp \id} \end{displaystyle} \oneover{1-\chi(\ell)/\ell}}{\begin{displaystyle} \prod_{\chi\in \gal(K^+/\rat)^*\setcomp \id} \end{displaystyle} \oneover{1-\chi(\ell)/\ell}} = \nu_\ell(K).
\]
Finally,  $\calo_K\units$ and $\calo_{K^+}\units $ agree up to torsion so
$R_K = 2R_{K^+}$, and $K^+$ is a real field so $\omega_{K^+}=2$.  Consequently,
\begin{align*}
\frac{h(K)}{h(K^+)}& =
\omega_K \oneover{4\pi^2}\sqrt{\frac{\abs{\Delta_K}}{\abs{\Delta_{K^+}}} }
\prod_\ell \nu_\ell(K)\\
&= \omega_K \nu_\infty(f) \prod_\ell \nu_\ell(f)
\end{align*}
by Proposition \ref{propmatchell} and Lemmas \ref{lemmatchp} and \ref{lemmatchinfinity}.
\end{proof}

In fact, \eqref{eqmain} has a natural interpretation in terms of abelian varieties. 

\begin{corollary}
For $f$ as in Theorem \ref{thmain}, suppose further that $\gal(K_f/\rat)$ is cyclic.  Then
\begin{equation}
\label{eqcormain}
\nu_\infty(f) \prod_\ell \nu_\ell(f) = \#\stk A_2(\ff_q;f),
\end{equation}
the number of isomorphism classes of principally polarized abelian surfaces over $\ff_q$ with characteristic polynomial of Frobenius $f$, weighted by (inverse) size of automorphism group.
\end{corollary}

\begin{proof}
If $(X,\lambda) \in \stk A_2(\ff_q;f)$, then $\#\aut(X,\lambda) =
\omega_K$.  By \cite{lauterCMAS}, $h(K)/h(K^+)$ is the (unweighted)
size of $\stk A_2(\ff_q;f)$.  Indeed, under the hypothesis
that $\gal(K_f/\rat)$ is cyclic of order $4$ and the maximality
condition \ref{enmax}, \cite[Thm.\ 3.1 and Cor.\ 3.2]{lauterCMAS}
show that the size of the isogeny class parametrized by $f$ is (in
their notation)
\[
\# {\mathfrak C}(K) =
\frac{h(K)}{h^+(K^+)}[(\calo_{K^+}\units)^+:\operatorname{N}_{K/K^+}(\calo_K\units)],
\]
where $h^+(K^+)$ denotes the narrow class group of $K^+$, and
$(\calo_{K^+}\units)^+$ denotes the totally positive units of
$\calo_{K^+}$.  Since 
\[
\null[(\calo_{K^+}\units)^+:\operatorname{N}_{K/K^+}(\calo_K\units)]
= \frac{h^+(K^+)}{h(K^+)},
\]
we find that $\#{\mathfrak C}(K) = h(K)/h(K^+)$.

 Now invoke Theorem \ref{thmain}.
\end{proof}

In fact, unpublished work of Howe shows that in much greater generality, $h(K)/h(K^+)$ computes the size of a suitable isogeny class.  Thus, we expect \eqref{eqcormain} to also hold when $K_f$ is biquadratic.

\bibliographystyle{abbrv}
\bibliography{bibliography}

\end{document}